\documentclass[12pt, a4]{amsart}
\usepackage{geometry}
\usepackage{amsxtra,amssymb,amsthm,amsmath,amscd,mathrsfs, epsfig, eufrak}
\usepackage{amscd, amsmath, mathrsfs, amssymb, amsthm, amsxtra, bbding, epsfig, eucal, eufrak, graphicx, latexsym, mathrsfs, mathbbol, bbold}
\usepackage[all]{xy}
\allowdisplaybreaks
\usepackage{tikz-cd}

\usepackage{tikz}
\usepackage[normalem]{ulem}
\usepackage{soul}
\usepackage{color}

\setstcolor{red}
\newcommand{\blue }[1]{{\color{blue}#1}}
\setstcolor{blue}
\makeatletter
\@namedef{subjclassname@2020}{%
	\textup{2020} Mathematics Subject Classification}
\makeatother
\usepackage{hyperref}
\hypersetup{hypertex=true,colorlinks=true,linkcolor=cyan,anchorcolor=cyan,citecolor=cyan}

\def \i {{\rm i}}

\def \d {\,{\rm d}}
\def\re{{\Re e\,}}
\def\im{{\Im m\,}}

\def\le{\leqslant}
\def\ge{\geqslant}

\theoremstyle{plain}
\newtheorem{theorem}{Theorem}[section]
\newtheorem{proposition}{Proposition}[section]
\newtheorem{lemma}[proposition]{Lemma}

\theoremstyle{remark}

\numberwithin{equation}{section}
\addtocounter{footnote}{1}

\numberwithin{equation}{section}

\addtocounter{footnote}{1}

\begin{document}
	
	\title[On the Distribution of large values of  $|\zeta(\sigma+{\rm i}t)|$]
	{On the Distribution of large values of  $|\zeta(\sigma+{\rm i}t)|$}
	\author[Zikang Dong]{Zikang Dong}
	\address{%
		CNRS LAMA 8050\\
		Laboratoire d'analyse et de math\'ematiques appliqu\'ees\\
		Universit\'e Paris-Est Cr\'eteil\\
		61 avenue du G\'en\'eral de Gaulle\\
		94010 Cr\'eteil Cedex\\
		France
	}
	\email{zikangdong@gmail.com}

	\date{\today}
	
	
\begin{abstract}
We investigate the distribution of large values of the Riemann zeta function $\zeta(s)$ in the strip $\frac{1}{2}<\re s<1$. For any fixed $\re s=\sigma\in(\frac{1}{2},1)$, we obtain an improved distribution function of large values of $|\zeta(\sigma+\i t)|$, holding in the same range as that given by Lamzouri.	
\end{abstract}

\maketitle

\section{Introduction}

Throughout this article, $\sigma$ will denote any fixed number in $(\frac{1}{2},1)$,
 {$\zeta(s)$ the Riemann zeta function} and $\log_j$ the $j$-th iterated logarithm. Firstly we make a brief review of the extreme values of $|\zeta(\sigma+\i t)|$ as $t$ varies. In 1928, Titchmarsh \cite{Ti28} showed that for any $\varepsilon>0$, we have
$$\limsup_{t\to\infty}\frac{\log |\zeta(\sigma+{\rm i}t)|}{{(\log t)}^{1-\sigma-\varepsilon}}=\infty.$$
In 1972, Levinson \cite{Le72} replaced $(\log t)^\varepsilon$ by $\log_2t$, by showing that for sufficiently large $T$ we have
$$\max_{t\in[0,T]}\log |\zeta(\sigma+{\rm i}t)|\gg\frac{(\log T)^{1-\sigma}}{\log_2T}.$$
In 1977, Montgomery \cite{Mon77} showed that 
\begin{align}
	\max_{t\in[0,T]}\log |\zeta(\sigma+{\rm i}t)|\ge \nu(\sigma)\frac{(\log T)^{1-\sigma}}{(\log_2T)^\sigma},\label{10}
\end{align}
where $\nu(\sigma)=\frac{1}{20}(\sigma-\frac{1}{2})^{1/2}$ unconditionally, 
and $\nu(\sigma)=\frac{1}{20}$ on assuming the Riemann hypothesis. This quantity ${(\log T)^{1-\sigma}}/{(\log_2T)^\sigma}$ is conjectured to be the true order of magnitude of 
$\max_{t\in  {[0, T]}}\log|\zeta(\sigma+\i t)|$. 
More precisely, we believe the following inequality holds:
$$
\max_{t\in [0, T]}\log |\zeta(\sigma+{\rm i}t)|
\asymp_\sigma \frac{(\log T)^{1-\sigma}}{(\log_2T)^\sigma}\cdot
$$ 
Thus, the only improvement of (\ref{10}) we could expect is to get larger values of $\nu(\sigma)$. We refer to \cite{A16,BS18}.

In 2011, applying a method of  {Granville} and Soundararajan \cite{GS06} to investigate the distribution of values of $|\zeta(1+\i t)|$, Lamzouri \cite{La2011} established the distribution of large values of $|\zeta(\sigma+\i t)|$ as $t$ varies in $[T,2T]$.	Let $T$ be sufficiently large. We define the distribution function by
\begin{equation}\label{defphi}
\Phi_T(\tau)
:= \frac {1}{ T}{\rm meas}\big\{t\in[T, 2T] : \log|\zeta(\sigma+{\rm i}t)|>\tau\big\}.
\end{equation}
Then there exists a positive constant $c(\sigma)$ such that we have
\begin{equation}\label{La2011}
\Phi_T(\tau)
= \exp\bigg(-(\tau\log^{\sigma}\tau)^{\frac{1}{1-\sigma}}
\bigg\{{\mathfrak a}_0
+ O\bigg(\frac1{\sqrt{\log\tau}} 
+ \bigg(\frac{(\tau\log\tau)^{\frac{1}{1-\sigma}}}{\log T}\bigg)^{\sigma-\frac{1}{2}}\bigg)\bigg\}\bigg)
\end{equation}
uniformly in the range $1\ll\tau\le c(\sigma)(\log T)^{1-\sigma}/\log_2T$, 
where ${\mathfrak a}_0$ will be defined later in (\ref{an2}). Despite the maximum of the range of $\tau$ being much less than (\ref{10}), the distribution function (\ref{La2011}) has more significance. If (\ref{La2011}) were to persist to the end of the viable range, then we could get a conjectural value of $\max_{t\in[T,2T]}\log |\zeta(\sigma+{\rm i}t)|$. More precisely, we have Lamzouri's conjecture (see \cite{La2011}):
$$
\max_{t\in[T,2T]}\log |\zeta(\sigma+{\rm i}t)|
= \{c(\sigma)+o(1)\} \frac{(\log T)^{1-\sigma}}{(\log_2T)^\sigma}$$
holds for $T\to\infty$, where
$$
c(\sigma) := \frac{C_0}{\sigma^{2\sigma}(1-\sigma)^{1-\sigma}}
$$
and $C_0$ will be defined in (\ref{cn2}). Note that this conjecture also implies the upper bound of $|\zeta(\sigma+\i t)|$. For more work concerning it, we refer to \cite{CS11,Cheng1999,Ford2002,Ri1967,Ti86}.

In this article, we aim to improve the distribution function (\ref{La2011}). We have a higher order expansion in the exponent, which is inspired by the work in \cite{Wu2007}. 

\begin{theorem}\label{th22}
Let $\sigma\in (\frac{1}{2}, 1)$ be a fixed real number.
	Let $\Phi_T(\tau)$ be defined in \eqref{defphi}. Then there exists a sequence of polynomials with real coefficients $\{\mathfrak{a}_n(\cdot)\}_{n\ge 0}$ with $\deg(\mathfrak{a}_n)\le n$, and a constant $c(\sigma)>0$, such that for any integer $N \ge 1$, we have
\begin{align*}
\Phi_T(\tau)
= \exp\bigg(\!-(\tau\log^{\sigma}\tau)^{\frac{1}{1-\sigma}}
\bigg\{\sum_{n=0}^{N}\frac{\mathfrak{a}_n(\log_2\tau)}{(\log\tau)^n}
+ O\bigg(\!\bigg(\frac{\log_2\tau}{\log\tau}\bigg)^{N+1} \!\!
+ \! \bigg(\frac{(\tau\log\tau)^{\frac{1}{1-\sigma}}}{\log T}\bigg)^{\sigma-\frac{1}{2}}\bigg)\bigg\}\bigg)
\end{align*}
uniformly for $T\to\infty$ and $1\ll\tau\le c(\sigma)(\log T)^{1-\sigma}/\log_2T$,
where the implied constant depend on $N$ and $\sigma$.
Especially, we have 
	\begin{align}
\mathfrak{a}_0
:= \bigg(\frac{\sigma^{2\sigma}}{C_0^{\sigma}(1-\sigma)^{2\sigma-1}}\bigg)^{1/(1-\sigma)}
\label{an2}
	\end{align}
	with $C_0$ defined in \eqref{cn2}.
\end{theorem}
The main new ingredient for the proof of Theorem \ref{th22} is Proposition \ref{prop4.1} below,
which gives a better approximation of the distribution function of the short Euler products:
\begin{equation}\label{def:PhiTtauy}
\Phi_T(\tau; y) := \frac{1}{T}  {\rm meas}\big\{t\in[T,2T]:\log|\zeta(\sigma+{\rm i}t; y)|>\tau\big\},
\end{equation}	
where
$$
\zeta(\sigma+\i t;y):=\prod_{p\le y}\bigg(1-\frac{1}{p^{\sigma+\i t}}\bigg)^{-1}.
$$
We refer to \cite{MM20} for similar work on $L$-functions attached to cusp forms.

\vskip 6mm
\section{ { Preliminary lemmas}}

Firstly, we will show the relationship between sums attached to the divisor function and the Bessel function by two asymptotic formulas. These will be used in the progress of calculating the moments of the short Euler products for the Riemann zeta function and the Dirichlet $L$-functions. One should pay attention that here $k$ is not necessarily an integer.

The modified Bessel function $I_0(t)$ of order $0$ is defined by
\begin{equation}\label{def:Bessel0}
I_0(t) := \int_0^1\exp(t\cos(2\pi\theta))\d\theta=\sum_{n=0}^{\infty} \frac{(t/2)^{
	 {2n}}}{(n!)^2}\cdot
\end{equation}
It's not difficult to see that
\begin{align}
\log I_0(t) 
& \ll t^2
\quad
(0\le t<1),
\label{23}
\\\noalign{\vskip 0,5mm}
\log I_0(t)
& \ll t 
\quad
(t\ge 1),
\label{24}
\\\noalign{\vskip 0,5mm}
(\log I_0(t))' 
& \ll {\rm min}\{1,|t|\}.
\label{25}
\end{align}

\begin{lemma}\label{l1}
Let $\sigma\in (\frac{1}{2}, 1)$ be a fixed real number.
For any prime $p$ and positive number $k$, we have
		\begin{align}
			\sum_{\nu\ge 0} \frac{d_{k/2}(p^{\nu})^2}{p^{2\nu\sigma}}
			& = I_0\bigg(\frac{k}{p^\sigma}\bigg) \exp\bigg\{O_{\sigma}\bigg(\frac k{p^{2\sigma}}\bigg)\bigg\},
			\label{21}
			\\\noalign{\vskip 0,5mm}
			\sum_{\nu\ge 0} \frac{d_{k/2}(p^{\nu})^2}{p^{2\nu\sigma}}
			& =\exp\bigg\{O_{\sigma}\bigg(\frac k{p^{\sigma}}\bigg)\bigg\}
			\qquad
			(p\le k^{ {1/\sigma}}),
			\label{22}
		\end{align}
		where  {the implied constants depend on $\sigma$ only}.
\end{lemma}

\begin{proof}
	See also of \cite[Lemma 4]{GS06}. 
Writing ${\rm e}(\theta):={\rm e}^{2\pi\i\theta}$, then 
\begin{align*}
\bigg|1-\frac{{\rm e}(\theta)}{p^\sigma}\bigg|^{-k}
& = \bigg(1-\frac{{\rm e}(\theta)}{p^\sigma}\bigg)^{-k/2} 
\bigg(1-\frac{{\rm e}(-\theta)}{p^\sigma}\bigg)^{-k/2}
\\
& = \sum_{\nu\ge0} \sum_{\nu'\ge0} 
\frac{d_{k/2}(p^\nu)d_{k/2}(p^{\nu'}){\rm e}((\nu-\nu')\theta)}{p^{(\nu+\nu')\sigma}}\cdot
\end{align*}
Thus we can derive that
\begin{align*}
\sum_{\nu\ge 0} \frac{d_{k/2}(p^\nu)^2}{p^{2\nu\sigma}} 
& = \int_0^1\bigg|1-\frac{{\rm e}(\theta)}{p^\sigma}\bigg|^{-2(k/2)}{\rm d}\theta
\\
& =\int_0^1\bigg(1-\frac{2\cos(2\pi\theta)}{p^\sigma}+\frac{1}{p^{2\sigma}}\bigg)^{-k/2}{\rm d}\theta
\\\noalign{\vskip 0,5mm}
& =\int_0^1\exp\bigg(-\frac{k}{2}\log\bigg(1-\frac{2\cos(2\pi\theta)}{p^\sigma}+\frac{1}{p^{2\sigma}}\bigg)\bigg){\rm d}\theta.
\end{align*}
This implies \eqref{21} thanks to the formula $\log(1+t)=t+O(t^2)\;(|t|\le 2^{-1/2})$,
and \eqref{22} follows from \eqref{21} and \eqref{24} immediately.
\end{proof}

\begin{lemma}\label{l122}
We have
$$
\sum_{p\le x}\frac{1}{p^\sigma}
= \frac{x^{ {1-\sigma}}}{(1-\sigma)\log x}+ O\bigg(\frac{x^{ {1-\sigma}}}{(1-\sigma)^{ {2}}(\log x)^2}\bigg)
$$
uniformly for $x\to\infty$ and  {$\frac{1}{2}<\sigma<1$,
where the implied constant is absolute}.
\end{lemma}

\begin{proof}
	This is equation (2.1) of \cite{La2011}. See also \cite[Lemma 6]{BS18},  \cite[Lemma 3.1]{Nor92}, and \cite[Lemma 3.3]{BG13}.
\end{proof}

We need to approximate Riemann zeta function $\zeta(s)$ by its short Euler product. 
The following lemma shows that when $\zeta(s)$ has no zero in a good region, 
it can be approximated well by its short Euler product.

\begin{lemma}\label{l142} 
 {Let $\sigma_0\in [\frac{1}{2}, 1)$ be a fixed number.}
Let $y\ge 2$ and $|t|\ge  y+3$ be real numbers and suppose that the rectangle 
$\{z : \sigma_0<\re z\le 1 \; \text{and} \;\, |\im z-t|\le y+2\}$ is free of zeros of $\zeta(z)$. Then for any $\sigma_0<\sigma\le 2$ and $|\xi-t|\le y$, we have
$$
|\log \zeta(\sigma+{\rm i}\xi)|
\ll (\log |t|)\log({\rm e}/(\sigma-\sigma_0)).
$$
		Further for $\sigma_0<\sigma\le 1$, we have
		$$\log \zeta(\sigma+{\rm i}t)=\sum_{n=2}^y \frac{\Lambda(n)}{n^{\sigma+{\rm i}t}\log n}+O\bigg(\frac{\log |t|}{(\sigma_1-\sigma_0)^2}y^{\sigma_1-\sigma}\bigg),$$
		where $\sigma_1 := {\rm min}\big(\sigma_0+(\log y)^{-1},\frac{1}{2}(\sigma+\sigma_0)\big)$. 
 {The implied constants depend on $\sigma_0$ at most.}
\end{lemma}

\begin{proof}
	See \cite[Lemma 1]{GS06}.
\end{proof}

With the help of Lemma \ref{l142}, 
as well as a result of zero density estimate for the Riemann zeta-function $\zeta(s)$, 
we can approximate $\zeta(s)$ by its short Euler product mostly often. 
Of course, here the short Euler product is a bit ``long", that means $y$ needs to be relatively large. 
Otherwise, the error term will be too large to make sense.

\begin{lemma}\label{l3}
 {Let $\sigma\in (\frac{1}{2}, 1)$ be a fixed number and 
$0<a(\sigma)<\frac{1}{2}(\sigma-\frac{1}{2})<2/(\sigma-\frac{1}{2})<A(\sigma)$.
Then for 
$$
T\to\infty
\quad\text{and}\quad
(\log T)^{A(\sigma)}\le y\le T^{a(\sigma)}
$$ 
the asymptotic formula
$$
\log\zeta(\sigma+{\rm i}t)=\sum_{n=2}^y \frac{\Lambda(n)}{n^{\sigma+{\rm i}t}\log n}
+ O\big(y^{-\frac{1}{2}(\sigma-\frac{1}{2})}(\log y)^2\log T\big)
$$
holds for all $t\in [T, 2T]$ except for a set of measure at most $O(T^{1-\frac{1}{2}(\sigma-\frac{1}{2})}y(\log T)^5)$,
where the implied constants depend on $\sigma$ at most}.
\end{lemma}

\begin{proof}
	This is essentially \cite[Lemma 2]{GS06} while we restrict $(\log T)^{A(\sigma)}\le  y\le T^{a(\sigma)}$ such that both the error term $O(y^{  {-\frac{1}{2}(\sigma-\frac{1}{2})}}(\log y)^2\log T)$ and the measure $T^{  {1-\frac{1}{2}(\sigma-\frac{1}{2})}}y(\log T)^5$ make sense. We replace the term $O(y^{  {-\frac{1}{2}(\sigma-\frac{1}{2})}}(\log T)^3)$ in \cite[Lemma 2]{GS06} by $O(y^{  {-\frac{1}{2}(\sigma-\frac{1}{2})}}(\log y)^2\log T)$. The proof has no difference from that of \cite[Lemma 2]{GS06}.
\end{proof}

In order to approximate $\zeta(s)$ by its ``shorter" Euler product, 
we need the following moment evaluation for the sum over complex power of primes 
between two large numbers $y$ and $z$, where $y$ can be relatively smaller.

\begin{lemma}\label{l4} 
 {Let $\sigma\in (\frac{1}{2}, 1)$ be a fixed number.} Then we have
$$
\frac1 T\int_T^{2T} \bigg|\sum_{y\le p\le z}\frac{1}{p^{\sigma+{\rm i}t}}\bigg|^{2k} {\rm d}t
\ll \bigg(k\sum_{y\le p\le z}\frac{1}{p^{2\sigma}}\bigg)^k + \frac{1}{T^  {\frac13}}
$$
for $2\le y\le z$ and all integers $1\le k\le (\log T)/(3\log z)$,
 {where the implied constant depends on $\sigma$ at most.} 
\end{lemma}

\begin{proof}
	This is \cite[Lemma 4.2]{La2011}.
\end{proof}

Using Lemma 2.4, we can give a generalization of Lemma 2.3. Here $y$ can be as small as $\log T$.

\begin{lemma}\label{l145}
Let $\sigma\in (\frac{1}{2}, 1)$ be a fixed number, and let $c_j(\sigma)$ be some suitable positive constants depending on $\sigma$.
Let $T\ge 2$, 
$\log T\le y\le(\log T)^{2/(\sigma- {\frac12})}$
and $c_1(\sigma)(\log_2T/\log T)^2\le \lambda \le (\log T)^ {\frac12}/(y^{\sigma- {\frac12}}\log_2T)$. 
Then we have
$$
|\log\zeta(\sigma+{\rm i}t) - \log\zeta(\sigma+{\rm i}t; y)|\le 2\lambda
$$
for all $t\in [T, 2T]$ except for a set of measure at most 
$O(T\exp(- {4{\rm e}^{-1}}(\sigma-\frac{1}{2})\lambda^2 y^{2\sigma-1}\log y))$.
\end{lemma}

\begin{proof}
Noticing that
\begin{align*}
\sum_{p\le z, \, p^{\nu}>z} \frac{1}{\nu p^{\nu\sigma}}
\le \sum_{p\le z, \, \nu\ge 2} \frac{(p^{\nu}/z)^{\sigma- {\frac12}}}{\nu p^{\nu\sigma}}
= \frac{1}{z^{\sigma- {\frac12}}} \sum_{p\le z, \, \nu\ge 2} \frac{1}{\nu p^{\nu/2}}
\ll \frac{1}{z^{\sigma- {\frac12}}} \sum_{p\le z} \frac{1}{p}
\ll \frac{\log_2z}{z^{\sigma- {\frac12}}},
\end{align*}
we can write
\begin{align*}
\sum_{2\le n\le z} \frac{\Lambda(n)}{n^{\sigma+{\rm i}t}\log n}
& = \sum_{p^{\nu}\le z} \frac{1}{\nu p^{\nu(\sigma+{\rm i}t)}}
= \sum_{p\le z} \sum_{\nu\ge 1} \frac{1}{\nu p^{\nu(\sigma+{\rm i}t)}}
+ O\Big(\frac{\log_2z}{z^{\sigma-1/2}}\Big)
\\
& = \log \zeta(1+{\rm i}t; z)
+ O\Big(\frac{\log_2z}{z^{\sigma- {\frac12}}}\Big).
\end{align*}
Using this and Lemma \ref{l3} with $y=z=(\log T)^{6/(\sigma- {\frac12})}$, we obtain 
$$
\log \zeta(1+{\rm i}t; z) + O\Big(\frac{\log_2z}{z^{\sigma- {\frac12}}}\Big)
= \log \zeta(1+{\rm i}t) + O\Big(\frac{(\log z)^2\log T}{z^{ {\frac12}(\sigma- {\frac12})}}\Big)
$$
i.e.
\begin{equation}\label{first-step}
\zeta(1+{\rm i}t)
= \zeta(1+{\rm i}t; z) 
\bigg\{1 + O\bigg(\bigg(\frac{\log_2T}{\log T}\bigg)^2\bigg)\bigg\}
\end{equation}
for all $t\in [T, 2T]$ but at most a set of measure of
\begin{equation}\label{exception:1}
T^{1-\frac{1}{2}(\sigma-\frac{1}{2})}z(\log T)^5
\ll T^{1-\frac{1}{4}(\sigma-\frac{1}{2})}.
\end{equation}
Then we use Lemma \ref{l4} to approximate $\zeta(\sigma+{\rm i}t;z)$ by $\zeta(\sigma+{\rm i}t;y)$ since 
$$
\zeta(\sigma+{\rm i}t;z)
= \zeta(\sigma+{\rm i}t;y)
\exp\bigg(\sum_{y\le p \le z}\bigg\{\frac{1}{p^{\sigma+{\rm i}t}}+O\bigg(\frac{1}{p^{2\sigma}}\bigg)\bigg\}\bigg).
$$
Choosing 
$$
k=\big\lfloor( {4{\rm e}^{-1}} (\sigma-\tfrac{1}{2})\lambda^2 y^{2\sigma-1}\log y\big\rfloor,
$$
which satisfies the condition in Lemma \ref{l4},
then by this lemma we have
$$
\frac{1}{T} \int_T^{2T}\bigg|\sum_{y\le p\le z}\frac{1}{p^{\sigma+{\rm i}t}}\bigg|^{2k}{\rm d}t
\ll  \bigg(k\sum_{y\le p\le z}\frac{1}{p^{2\sigma}}\bigg)^k + \frac{1}{T^ {\frac13}}
\ll  \bigg(\frac{k}{(\sigma-\frac{1}{2})y^{2\sigma-1}\log y}\bigg)^k + \frac{1}{T^ {\frac13}}\cdot
$$
So the frequency of $t\in[T,2T]$ such that 
$|\log\zeta(\sigma+{\rm i}t; z) - \log\zeta(\sigma+{\rm i}t; y)|>2\lambda$ is less than
\begin{equation}\label{exception:2}
\frac{1}{T} \int_T^{2T}\bigg|\frac{1}{2\lambda}\sum_{y\le p\le z}\frac{1}{p^{\sigma+{\rm i}t}}\bigg|^{2k}{\rm d}t
\ll \bigg(\frac{k}{4(\sigma-\frac{1}{2})\lambda^2 y^{2\sigma-1}\log y}\bigg)^k+(2\lambda)^{-2k}T^ {-\frac13}.
\end{equation}
Since $\lambda>c_1(\sigma)(\log_2T/\log T)^2$, we have
\begin{align*}
& |\log\zeta(\sigma+{\rm i}t) - \log\zeta(\sigma+{\rm i}t; y)|
\\
& \ge |\log\zeta(\sigma+{\rm i}t; z) - \log\zeta(\sigma+{\rm i}t; y)|
- |\log\zeta(\sigma+{\rm i}t; z) - \log\zeta(\sigma+{\rm i}t)|
\\
& \ge 2\lambda + O((\log_2)^2/(\log T)^2)
> \lambda.
\end{align*}
By \eqref{exception:1} and \eqref{exception:2},
the frequency of $t\in[T,2T]$ such that 
$|\log\zeta(\sigma+{\rm i}t) - \log\zeta(\sigma+{\rm i}t; y)|>2\lambda$ is less than, thanks to our choice of $k$,
\begin{align*}
& \ll \bigg(\frac{k}{4(\sigma-\frac{1}{2})\lambda^2 y^{2\sigma-1}\log y}\bigg)^k
+ \frac{1}{(2\lambda)^{2k}T^ {\frac13}}
+ \frac{1}{T^{\frac{1}{4}(\sigma-\frac{1}{2})}}
\\
& \ll {\rm e}^{-k}
+ (2\lambda)^{-2k} T^{-\frac{1}{4}(\sigma-\frac{1}{2})}.
\end{align*}
This implies the required result, since our hypothesis on $(\lambda, y)$ garanties 
$$
(2\lambda)^{-2k} T^{-\frac{1}{4}(\sigma-\frac{1}{2})}
\le T^{-\frac{1}{8}(\sigma-\frac{1}{2})}
\le \exp(- {4{\rm e}^{-1}}(\sigma-\tfrac{1}{2})\lambda^2 y^{2\sigma-1}\log y).
$$
	Combining this with the first step, Lemma \ref{l145} follows.
\end{proof}	

\vskip 6mm

\section{{ Moments of the short Euler products}}

In this section, we will evaluate the $k$-th moment of the short Euler product $\zeta(\sigma+\i t;y)$ by proving the following proposition, which is important for the proof of Theorem \ref{th22}. It has a higher order expansion in the exponent, which is an improvement of equation (4.2) in \cite{La2011}.

\begin{proposition}\label{th1} 
 {Let $\sigma\in (\frac{1}{2}, 1)$ be a fixed constant and let $N$ be a non-negative integer.}
Then we have
$$
\frac{1}{T}\int_T^{2T}|\zeta(\sigma+{\rm i}t;y)|^k{\rm d}t
= \exp\bigg(\frac{k^{1/\sigma}}{\log k}\bigg\{\sum_{n=0}^N\frac{C_n}{(\log k)^n}
+ O\bigg(\frac1{(\log k)^{N+1}} + \bigg(\frac {k^{1/\sigma}}{y}\bigg)^{2\sigma-1}\bigg)\bigg\}\bigg)
$$
uniformly for
\begin{equation}\label{Cond:Tyk}
T\ge 3
\qquad\text{and}\qquad
ky^{1-\sigma} \le \tfrac{1}{8}(1-\sigma)\log T,
\end{equation} 
where 
\begin{equation}\label{cn2}
C_n := \int_0^\infty \frac{(\log t)^n}{t^{{1}/{\sigma}+1}} \log I_0(t) \, {\rm d}t
\quad(n\ge0)
\end{equation}
and $I_0(t)$ is the Bessel function given by \eqref{def:Bessel0}. 
Especially, we have $C_0>0$.
\end{proposition}

The integer $n\ge 1$ is called $y$-friable if the largest prime factor $P(n)$ of $n$ is less than $y$
($P(1)=1$ by convention).
Denote by $S(y)$ the set of $y$-friable integers.
We will show that, in the expansion of the $k$-th moment of $\zeta(\sigma+{\rm i}t;y)$, the diagonal terms lead to the main term, while the off-diagonal terms contribute to the error term. Again we strengthen that, $k$ is not necessarily an integer.

\begin{lemma}\label{l30} 
Let $\sigma\in (\frac{1}{2}, 1)$ be a fixed constant.
Then we have
$$
\frac{1}{T}\int_T^{2T}|\zeta(\sigma+{\rm i}t;y)|^k{\rm d}t
= \sum_{n\in S(y)}\frac{d_{k/2}(n)^2}{n^{2\sigma}}+O\bigg(\exp\bigg(-\frac{\log T}{4\log y}\bigg)\bigg),
$$ 
uniformly for $(T, y, k)$ in \eqref{Cond:Tyk},
where the implied constant depends on $\sigma$ only.
\end{lemma}

\begin{proof}
	This is a special case of Proposition 4.1 of \cite{La2011}.
\end{proof}

Now we are ready to prove Proposition \ref{th1}.

\vskip 0,5mm

\noindent{\it Proof of Proposition \ref{th1}.}
In view of Lemma \ref{l30}, it is sufficient to show that
\begin{equation}\label{Asymp:D}
\sum_{n\in S(y)}\frac{d_{k/2}(n)^2}{n^{2\sigma}}
= \exp\bigg(\frac{k^{1/\sigma}}{\log k} \bigg\{\sum_{n=0}^N\frac{C_n}{(\log k)^n}
+ O\bigg(\frac1{(\log k)^{N+1}} + \bigg(\frac {k^{1/\sigma}}{ y}\bigg)^{2\sigma-1}\bigg)\bigg\}\bigg).
\end{equation}

Firstly, we note that \eqref{Asymp:D} is trivial if $y\le k^{1/\sigma}$.
In fact, since the divisor function is multiplicative, by \eqref{22} of Lemma \ref{l1} we have
\begin{align*}
\sum_{n\in S(y)}\frac{d_{k/2}(n)^2}{n^{2\sigma}}
= \prod_{p\le  y} \sum_{\nu\ge 0}\frac{d_{k/2}(p^{\nu})^2}{p^{2\nu\sigma}}
= \exp\bigg\{O\bigg(\sum_{p\le k^{1/\sigma}}\frac k{p^{\sigma}}\bigg)\bigg\}
= \exp\bigg\{O\bigg(\frac{k^{1/\sigma}}{\log k}\bigg)\bigg\}.
\end{align*}

Now we treat the case of $y>k^{1/\sigma}$.
As before, by Lemma \ref{l1} we can write
\begin{equation}\label{D:expression}
\begin{aligned}
\sum_{n\in S(y)}\frac{d_{k/2}(n)^2}{n^{2\sigma}}
& = \prod\limits_{p\le  k^{1/(2\sigma)}} \exp\bigg\{O\bigg(\frac k{p^{\sigma}}\bigg)\bigg\}
\prod_{k^{1/(2\sigma)}<p\le y}I_0\bigg(\frac{k}{p^\sigma}\bigg) \exp\bigg\{O\bigg(\frac k{p^{2\sigma}}\bigg)\bigg\}
\\
& = \exp\bigg\{O_{\sigma, N}\bigg(\frac{k^{1/\sigma}}{(\log k)^{N+2}}\bigg)\bigg\}
\prod_{ k^{1/(2\sigma)}<p\le y}I_0\bigg(\frac{k}{p^\sigma}\bigg),
\end{aligned}
\end{equation}
where the last equation holds since
$$
\sum_{p\le  k^{1/(2\sigma)}} \frac k{p^{\sigma}} 
\ll k \frac{(k^{1/(2\sigma)})^{1-\sigma}}{\log k^{1/(2\sigma)}}
\ll \frac{k^{1/2+1/(2\sigma)}}{\log k}
\ll_{\sigma, N} \frac{k^{1/\sigma}}{(\log k)^{N+2}}
$$
and
$$
\sum_{k^{1/(2\sigma)}<p\le y}\frac{k}{p^{2\sigma}}
\ll \frac{k^{1/(2\sigma)}}{\log k} 
\ll_{\sigma, N} \frac{k^{1/\sigma}}{(\log k)^{N+2}}\cdot
$$


Next we evaluate the second factor on the right-hand side de \eqref{D:expression}.
Taking the logarithm of this factor and using the prime number theorem, we have
\begin{equation}\label{814}
\log\prod_{k^{1/(2\sigma)}<p\le y} I_0\bigg(\frac{k}{p^\sigma}\bigg)
= \int_{k^{1/(2\sigma)}}^y\log I_0\bigg(\frac {k}{u^\sigma}\bigg){\rm d}\pi(u)
= \mathcal{M} + \mathcal{E},
\end{equation}
where
$$
\mathcal{M} := \int_{k^{1/(2\sigma)}}^y\log I_0\bigg(\frac {k}{u^\sigma}\bigg)\frac{{\rm d}u}{\log u},
\qquad
\mathcal{E} := \int_{k^{1/(2\sigma)}}^y\log I_0\bigg(\frac {k}{u^\sigma}\bigg){\rm d}O\big(u{\rm e}^{-c\sqrt{\log u}}\big).
$$

In view of \eqref{23} and \eqref{24}, we always have $\log I_0(t) \ll t^2\;(t\ge 0)$.
Thus using this bound and \eqref{25}, we can derive that
\begin{equation}\label{UB:E}
\begin{aligned}
\mathcal{E}
& = \log I_0\bigg(\frac {k}{u^\sigma}\bigg)O\big(u{\rm e}^{-c\sqrt{\log u}}\big)\bigg|_{k^{1/(2\sigma)}}^y
- \int_{k^{1/2\sigma}}^y\bigg(\log I_0\bigg(\frac {k}{u^\sigma}\bigg)\bigg)'
O\big(u{\rm e}^{-c\sqrt{\log u}}\big){\rm d}u
\\
& \ll \bigg(\frac{k}{y^\sigma}\bigg)^2 \frac{y}{{\rm e}^{c'\sqrt{\log y}}}
+ \frac{k^{1/2+1/(2\sigma)}}{{\rm e}^{c'\sqrt{\log k}}}
+ k \int_{k^{1/(2\sigma)}}^{k^{1/\sigma}} \frac{{\rm e}^{-c\sqrt{\log u}}}{u^{\sigma}}{\rm d}u
+ k^2 \int_{k^{1/\sigma}}^y \frac{{\rm e}^{-c\sqrt{\log u}}}{u^{2\sigma}}{\rm d}u
\\
& \ll \bigg(\frac{k^{1/\sigma}}{y}\bigg)^{2\sigma-1} k^{1/\sigma} {\rm e}^{-c'\sqrt{\log y}}
+ k^{1/2+1/(2\sigma)}{\rm e}^{-c'\sqrt{\log k}}
+ k^{1/\sigma}{\rm e}^{-c'\sqrt{\log k}}
\\
& \ll \frac{k^{1/\sigma}}{\log k} 
\bigg(\bigg(\frac{k^{1/\sigma}}{y}\bigg)^{2\sigma-1} \frac{\log k}{{\rm e}^{c'\sqrt{\log y}}}
+ \frac{k^{-(1/\sigma-1)/2}\log k}{{\rm e}^{c'\sqrt{\log k}}}\bigg).
\end{aligned}
\end{equation}
This is acceptable, since $y\ge k^{1/\sigma}$.

In order to calculate the main term of \eqref{814}, setting $t=k/u^\sigma$, and integrating by substitution, 
then we have
$$
\mathcal{M}
= k^{1/\sigma} \int_{k/y^\sigma}^{k^{1/2}}\frac{\log I_0(t)}{t^{1/\sigma+1}\log(k/t)}{\rm d}t
= \frac{k^{1/\sigma}}{\log k} \int_{k/y^\sigma}^{k^{1/2}}\frac{\log I_0(t)}{t^{1/\sigma+1}}
\frac 1{1-\log t/\log k} {\rm d}t.
$$
For $k/y^\sigma\le t\le k^{1/2}$, we can write
$$
\frac{1}{1-\log t/\log k}
= \sum_{n=0}^{N}\frac{(\log t)^n}{(\log k)^n}
+ O_{\sigma, N}\bigg(\frac{(\log t)^{N+1}}{(\log k)^{N+1}}\bigg).
$$
Thus
$$
\mathcal{M}
= \frac{k^{1/\sigma}}{\log k}
\bigg\{\sum_{n=0}^{N} \frac{C_n(k, y)}{(\log k)^n} + O\bigg(\frac{1}{(\log k)^{N+1}}\bigg)\bigg\},
$$
where 
$$
C_n(k, y) := \int_{k/y^\sigma}^{k^{1/2}}\frac{(\log t)^n}{t^{1/\sigma+1}} \log I_0(t)\, {\rm d}t
$$
and we have used \eqref{23}-\eqref{24} to bound
$$
\int_{k/y^\sigma}^{k^{1/2}} \frac{(\log t)^{N+1}}{t^{1/\sigma+1}} \log I_0(t)\, {\rm d}t
\ll \int_{k/y^\sigma}^1 \frac{(\log t)^{N+1}}{t^{1/\sigma-1}} {\rm d}t
+ \int_1^{k^{1/2}} \frac{(\log t)^{N+1}}{t^{1/\sigma}} {\rm d}t
\ll_{\sigma, N} 1.
$$
On the other hand, 
we enlarge the integral interval to $(0, \infty)$, and use the definition of $C_n$, then the main term of the last formula is 
$$
C_n(k, y) = C_n - \widetilde{C}_n
$$
where
\begin{align*}
\widetilde{C}_n
& := \bigg(\int_{0}^{k/y^\sigma}+\int_{k^{1/2}}^{\infty}\bigg)\frac{(\log t)^n}{t^{1/\sigma+1}} \log I_0(t) {\rm d}t
\\
& \; \ll \int_{0}^{k/y^\sigma} \frac{(-\log t)^n}{t^{1/\sigma-1}} {\rm d}t
+ \int_{k^{1/2}}^{\infty}\frac{(\log t)^n}{t^{1/\sigma}}{\rm d}t
\\
& \; \ll_{\sigma, N} \bigg(\frac{k^{1/\sigma}}{y}\bigg)^{2\sigma-1}(\log(y^\sigma/k))^n
+ \frac{(\log k)^n}{k^{(1/\sigma-1)/2}},
\end{align*}
thanks to \eqref{23}-\eqref{24}.
It follows that 
\begin{align*}
\frac{\widetilde{C}_n}{(\log k)^n}
& \ll \bigg(\frac{k^{1/\sigma}}{y}\bigg)^{2\sigma-1} \bigg(\frac{\log(y^\sigma/k)}{\log k}\bigg)^n
+ \frac{1}{k^{(1/\sigma-1)/2}}
\\
& \ll \bigg(\frac{k^{1/\sigma}}{y}\bigg)^{2\sigma-1}
+ \frac{1}{(\log k)^{N+1}},
\end{align*}
since $\big(\frac{\log(y^\sigma/k)}{\log k}\big)^n\ll_{\sigma, N} 1$ if $y\le k^{2/\sigma}$ and otherwise we have
\begin{align*}
 \bigg(\frac{k^{1/\sigma}}{y}\bigg)^{2\sigma-1}\bigg(\frac{\log(y^\sigma/k)}{\log k}\bigg)^n\le \bigg(\frac{k^{1/\sigma}}{k^{2/\sigma}}\bigg)^{2\sigma-1}\bigg(\frac{\log((k^{2/\sigma})^\sigma/k)}{\log k}\bigg)^n\ll \frac{1}{(\log k)^{N+1}}\cdot
\end{align*}
Thus
\begin{equation}\label{Asymp:M}
\mathcal{M}
= \frac{k^{1/\sigma}}{\log k}
\bigg\{\sum_{n=0}^{N} \frac{C_n}{(\log k)^n} 
+ O_{\sigma, N}\bigg(\frac{1}{(\log k)^{N+1}}
+ \bigg(\frac{k^{1/\sigma}}{y}\bigg)^{2\sigma-1}\bigg)\bigg\}.
\end{equation}
Now the required \eqref{Asymp:D} follows from \eqref{D:expression}, \eqref{814}, \eqref{UB:E} and \eqref{Asymp:M}.
\hfill
$\square$

\vskip 6mm

\section{Proof of Theorem \ref{th22}}

Recall that we have define the short Euler products by
$$
\zeta(\sigma+\i t;y):=\prod_{p\le y}\bigg(1-\frac{1}{p^{\sigma+\i t}}\bigg)^{-1},
$$
and its distribution function
$$
\Phi_T(\tau; y) := \frac{1}{T}  {\rm meas}\big\{t\in[T,2T]:\log|\zeta(\sigma+{\rm i}t; y)|>\tau\big\}.
$$	
In this section, we aim to prove the following proposition.

\begin{proposition}\label{prop4.1}
Let $\sigma\in (\frac{1}{2}, 1)$ be a fixed constant and let $N\ge1$ be an integer
and let $c_0=c_0(\sigma, N)$ be a large positive constant depending on $(\sigma, N)$.
Then we have
\begin{align*}
\Phi_T((1+\varepsilon_0)\tau; y)
\le \exp\bigg(\!-(\tau \log^{\sigma}\tau)^{\frac{1}{1-\sigma}}
\bigg\{\sum_{n=0}^{N}\frac{\mathfrak{a}_n(\log_2\tau)}{(\log\tau)^n} + O(\varepsilon_0)\bigg\}\bigg)
\le \Phi_T((1-\varepsilon_0)\tau; y)
\end{align*}
uniformly for 
\begin{equation}\label{Cond-Prop4.1:Tytau}
T\to\infty,
\quad
\log T\le y\le{(\log T)^2},
\quad 
1\ll\tau\le { \frac{c(\sigma)}{\log_2T}\bigg(\frac{\log T}{y^{1-\sigma}}\bigg)^{\frac{1-\sigma}{\sigma}}},
\end{equation}
where $c(\sigma)$ is a positive constant depending only on $\sigma$,
\begin{equation}\label{def:varepsilon0}
\varepsilon_0
= \varepsilon_0(\tau, y) 
= c_0 \bigg\{\bigg(\frac{\log_2\tau}{\log\tau}\bigg)^{N+1}
+ \bigg(\frac{(\tau\log\tau)^{\frac{1}{1-\sigma}}}{y}\bigg)^{\sigma-\frac{1}{2}}\bigg\},
\end{equation} 
the polynomials $\mathfrak{a}_n(\cdot)$ is the same as in Theorem \ref{th22} and the implied constant is absolute.
\end{proposition}

\subsection{Two preliminary lemmas}
The following lemma relates the moments of the short Euler products to the distribution function.

\begin{lemma}\label{l821} 
Let $\sigma\in (\frac{1}{2}, 1)$ be a fixed constant.
For any non-negative integer $N$, we have
$$
\int_{-\infty}^{\infty}\Phi_T(t;y)k{\rm e}^{kt} {\rm d}t
= \exp\bigg(\frac{k^{1/\sigma}}{\log k}\bigg\{\sum_{n=0}^N\frac{C_n}{(\log k)^n}
+ O\bigg(\frac1{(\log k)^{N+1}} + \bigg(\frac {k^{1/\sigma}}{ y}\bigg)^{2\sigma-1}\bigg)\bigg\}\bigg)
$$
uniformly for $(T, y, k)$ in \eqref{Cond:Tyk},
where $C_n$ is defined in \eqref{cn2} and the implied constant depends only on $N$ and $\sigma$. 
\end{lemma}
\begin{proof}
	Since
	\begin{align*}
		\int_{-\infty}^{\infty} \Phi_T(u; y) k{\rm e}^{ku} {\rm d} u
		& = \frac{1}{T} \int_{-\infty}^{\infty}
		\bigg(\mathop{\int_T^{2T}}_{\log|\zeta(\sigma+{\rm i}t; y)|>u} 1 \d t\bigg) k{\rm e}^{ku} {\rm d} u
		\\
		& = \frac{1}{T} \int_T^{2T}
		\bigg(\int_{-\infty}^{ \log|\zeta(\sigma+{\rm i}t; y)|} k{\rm e}^{ku} \d u\bigg) {\rm d} t
		\\\noalign{\vskip 1mm}
		& = \frac{1}{T} \int_T^{2T}   |\zeta(\sigma+{\rm i}t; y)|)^{k} {\rm d}t,
	\end{align*}
the required result of Lemma \ref{l821} follows from Proposition \ref{th1} immediately.
\end{proof}

\begin{lemma}\label{l822}
Let $\sigma\in (\frac{1}{2}, 1)$ be a fixed constant. 
Let $\{a_n\}_{n\ge0}$ be a sequence of real numbers and $N\ge0$ be an integer. If 
\begin{equation}\label{822}
\tau = \frac{k^{1/\sigma-1}}{\sigma\log k}\sum_{n=0}^{N+1}\frac{a_n}{(\log k)^n}
\qquad
(k\to\infty),
\end{equation}
then there is a sequence of polynomials $\{b_n(\cdot)\}_{n\ge0}$ 
with $\deg(b_n)\le n$ and $b_{0}=\frac{\sigma}{1-\sigma}$ such that 
\begin{equation}\label{822bis}
\log k
= (\log\tau) 
\bigg\{\sum_{n=0}^{N}\frac{b_n(\log_2\tau)}{(\log \tau)^n}
+ O\bigg(\bigg(\frac{\log_2\tau}{\log\tau}\bigg)^{N+1}\bigg)\bigg\},
\end{equation}
where the implied constant depends on the sequence $\{a_n\}_{n\ge0}$ and $N$.
\end{lemma}

\begin{proof}
We prove it by recurrence. Firstly, taking logarithm of both sides in (\ref{822}), we have 
\begin{equation}\label{823}
\log\tau 
= \frac{1-\sigma}{\sigma}\log k-\log\sigma-\log_2k+\log\bigg(\sum_{n=0}^{N+1}\frac{a_n}{(\log k)^n}\bigg).\end{equation}
From this we derive that
$$
\log(\tau\log\tau) =\frac{1-\sigma}{\sigma}\log k+O_{\sigma}(1)
$$
and
\begin{equation}\label{k:tau}
\log k 
= \frac{\sigma}{1-\sigma}\log(\tau\log\tau)+O_{\sigma}(1)
= \frac{\sigma}{1-\sigma}(\log\tau) \bigg(1 + \frac{\log_2\tau+O_{\sigma}(1)}{\log \tau}\bigg),
\end{equation}
which is the case for $N=0$. 

Now assume we already have
$$
\log k
= (\log\tau) \bigg\{\sum_{n=0}^{m}\frac{b_n(\log_2\tau)}{(\log \tau)^n}
+O\bigg(\bigg(\frac{\log_2\tau}{\log\tau}\bigg)^{m+1}\bigg)\bigg\},
$$
for some $m< N$. 
Inserting this into \eqref{823}, it follows that 
\begin{align*}
\log\tau 
& = \frac{1-\sigma}{\sigma}\log k-\log\sigma-\log\bigg((\log\tau)\bigg\{\sum_{n=0}^{m}\frac{b_n(\log_2\tau)}{(\log \tau)^n}+O\bigg(\bigg(\frac{\log_2\tau}{\log\tau}\bigg)^{m+1}\bigg)\bigg\}\bigg)
\\
& \quad 
+ \log\bigg(\sum_{n=0}^{N+1}{a_n}{\bigg((\log\tau)\bigg\{\sum_{n=0}^{m}\frac{b_n(\log_2\tau)}{(\log \tau)^n}
+ O\bigg(\bigg(\frac{\log_2\tau}{\log\tau}\bigg)^{m+1}\bigg)\bigg\}\bigg)^{-n}}\bigg),
\end{align*}
from which we derive that
\begin{align*}
\frac{1-\sigma}{\sigma}\log k 
& = (\log\tau)\bigg\{1
+ \frac{\log\sigma+\log_2\tau}{\log\tau} 
+ \frac{1}{\log\tau}\log\bigg\{\sum_{n=0}^{m}\frac{b_n(\log_2\tau)}{(\log \tau)^n}+ O\bigg(\Big(\frac{\log_2\tau}{\log\tau}\Big)^{m+1}\bigg)\bigg\}
\\
& \quad 
- \frac{1}{\log\tau}\log\bigg(\sum_{n=0}^{N+1}{a_n}{\bigg((\log\tau)\bigg\{\sum_{n=0}^{m}\frac{b_n(\log_2\tau)}{(\log \tau)^n}+O\bigg(\Big(\frac{\log_2\tau}{\log\tau}\Big)^{m+1}\bigg)\bigg\}\bigg)^{-n}}\bigg)\bigg\}.
\end{align*}
By expansion of the log-terms, we can obtain
$$
\frac{1-\sigma}{\sigma}\log k
= (\log\tau) \bigg\{\sum_{n=0}^{m+1} \frac{b_n^*(\log_2\tau)}{(\log \tau)^n}
+ O\bigg(\bigg(\frac{\log_2\tau}{\log\tau}\bigg)^{m+2}\bigg)\bigg\}
$$
with some polynomials $b_n^*(\log_2\tau)$ of $\deg(b_n^*)\le n$ and of $b_0^*=1$.	Thus Lemma \ref{l822} follows from recurrence.
\end{proof}

\subsection{Proof of Proposition \ref{prop4.1}}
Let $\{a_n\}_{n\ge 0}$ be a real sequence depending on $\sigma$, 
which will be chosen later.
It is clear that there is a large constant $t_0=t_0(\sigma)$ such that the function
$$
t\mapsto \frac{t^{1/\sigma-1}}{\sigma\log t}\sum_{n=0}^{N+1}\frac{a_n}{(\log t)^n}
$$
is strictly increasing on $[t_0, \infty)$.
Thus we choose a unique $k$ such that 
\begin{equation}\label{relation:tau=k}
\tau = \frac{k^{1/\sigma-1}}{\sigma\log k}\sum_{n=0}^{N+1}\frac{a_n}{(\log k)^n}\cdot
\end{equation} 
Noticing that \eqref{k:tau} and \eqref{Cond-Prop4.1:Tytau} imply that
$$
ky^{1-\sigma}
\ll_{\sigma} (\tau\log\tau)^{\frac{\sigma}{1-\sigma}} y^{1-\sigma}
\ll_{\sigma}  {\frac{\log T}{y^{1-\sigma}} y^{1-\sigma}}
\le \tfrac{1}{8}(1-\sigma)\log T,
$$
we can apply Lemma \ref{l821} to write
\begin{equation}\label{eq4.3}
\int_{-\infty}^{\infty}\Phi_T(t;y)k{\rm e}^{kt} {\rm d}t
= \exp\bigg(\frac{k^{1/\sigma}}{\log k}\bigg\{\sum_{n=0}^{2N+1}\frac{C_n}{(\log k)^n} 
+ O_{\sigma, N}(R_{2N+2}(k, y))\bigg\}\bigg),
\end{equation}
where 
$$
R_{2N+2}(k, y)
:= \frac1{(\log k)^{2N+2}} + \bigg(\frac {k^{1/\sigma}}{ y}\bigg)^{2\sigma-1}.
$$ 
We choose 
\begin{equation}\label{defepsilon}
\varepsilon 
= A\bigg(\frac{1}{(\log k)^{N+1}} + \bigg(\frac{k^{1/\sigma}}{y}\bigg)^{\sigma-\frac{1}{2}}\bigg)
\in (0, 10^{-2022})
\qquad
(k\ge k_0),
\end{equation} 
where $A=A(\sigma, N)$ and $k_0=k_0(\sigma, N)$ are large constants depending on $(\sigma, N)$, and let 
$$
k_1 := (1+\varepsilon)k,
\quad
k_2 := (1-\varepsilon)k,
\quad
\tau_1 := \Big(1+\frac{\varepsilon}{ 2\sigma}\Big)\tau,
\quad
\tau_2 := \Big(1-\frac{\varepsilon}{2\sigma}\Big)\tau.
$$
When $t\le \tau_2$, we have
$$
kt\le(k-k_2)(\tau_2-t)+kt=(k-k_2)\tau_2+k_2t=\varepsilon k\tau_2+k_2t.$$
Thus
\begin{equation}\label{eq4.5}
\int_{-\infty}^{\tau_2} {\rm e}^{kt}\Phi_T(t;y){\rm d}t
\le {\rm e}^{\varepsilon k\tau_2}\int_{-\infty}^{\infty}{\rm e}^{k_2t}\Phi_T(t;y){\rm d}t.
\end{equation}
Using \eqref{eq4.3} and noticing that $R_{2N+2}(k_2, y)\ll_{\sigma, N} R_{2N+2}(k, y)$, we have
$$
\int_{-\infty}^{\infty} {\rm e}^{k_2t}\Phi_T(t;y){\rm d}t
 = \exp\bigg(\frac{k_2^{1/\sigma}}{\log k_2}
 \bigg\{\sum_{n=0}^{2N+1}\frac{C_n}{(\log k_2)^n}+ O_{\sigma, N}(R_{2N+2}(k, y))\bigg\}\bigg).
$$
Inserting this into (\ref{eq4.5}) and using the definition of $\tau_2$ with \eqref{relation:tau=k}, then we have
\begin{equation}\label{functionf}
\int_{-\infty}^{\tau_2} {\rm e}^{kt}\Phi_T(t;y){\rm d}t
\le \exp\bigg(\frac{k^{1/\sigma}}{\log k}\{\mathcal{S}_1+\mathcal{S}_2 + O_{\sigma, N}(R_{2N+2}(k, y))\}\bigg),
\end{equation}
where
\begin{align*}
\mathcal{S}_1
& := \frac{\varepsilon}{\sigma} \Big(1-\frac{\varepsilon}{2\sigma}\Big)
\sum_{n=0}^{N+1}\frac{a_n}{(\log k)^n},
\\
\mathcal{S}_2
& := \frac{(1-\varepsilon)^{1/\sigma}}{1+\log (1-\varepsilon)/\log k}
\sum_{n=0}^{2N+1}\frac{C_n}{(\log k)^n}\frac{1}{(1+\log(1-\varepsilon)/\log k)^n}\cdot
\end{align*}
The first part $\mathcal{S}_1$ can be calculated easily, using the choice of $\varepsilon$, as
\begin{equation}\label{825}
\mathcal{S}_1
= \frac{\varepsilon}{\sigma}\sum_{n=0}^{N+1}\frac{a_n}{(\log k)^n}
- \frac{\varepsilon^2}{2\sigma^2}a_0 + o_{\sigma, N}(R_{2N+2}(k, y)).
\end{equation}
In order to calculate the second part $\mathcal{S}_2$, we take Taylor series for  {$\log(1-\varepsilon)$}, use the geometric sries formula, and put all infinitesimal of higher order than $R_{2N+2}$ into the error term, then we have
\begin{align*}
\mathcal{S}_2
& = \bigg(1-\frac{\varepsilon}{\sigma}+\frac{\varepsilon^2}{2\sigma}\bigg(\frac{1}{\sigma}-1\bigg)\bigg)
\bigg(1+\frac{\varepsilon}{\log k}\bigg)
\sum_{n=0}^{2N+1}\frac{C_n}{(\log k)^n}{\bigg(1+\frac{n\varepsilon}{\log k}\bigg)}
 + o_{\sigma, N}(R_{2N+2}(k, y))
\\
& = \bigg(1 - \frac{\varepsilon}{\sigma}+\frac{\varepsilon}{\log k}
+ \frac{\varepsilon^2}{2\sigma^2} - \frac{\varepsilon^2}{2\sigma}\bigg)
\sum_{n=0}^{2N+1}\frac{C_n}{(\log k)^n}{\bigg(1+\frac{n\varepsilon}{\log k}\bigg)}
+ o_{\sigma, N}(R_{2N+2}(k, y)).
\end{align*} 
We separate the same part as in the exponent of (\ref{825}) from the above formula, and again put all the infinitesimal of higher order than $R_{2N+2}$ into the error term, then we can write
\begin{equation}\label{eq4.7}
\mathcal{S}_2 
= \sum_{n=0}^{2N+1}\frac{C_n}{(\log k)^n}
+ \frac{\varepsilon^2}{2\sigma^2}C_0-\frac{\varepsilon^2}{2\sigma}C_0
+ -\frac{\varepsilon}{\sigma}
+ \frac{\varepsilon}{\sigma}C_0\sum_{n=1}^{N+1}\frac{\sigma nC_{n-1}-C_n}{(\log k)^n}
+ o_{\sigma, N}(R_{2N+2}(k, y)).
\end{equation}
Combining \eqref{825} and \eqref{eq4.7}, we have
\begin{align*}
\mathcal{S}_1 + \mathcal{S}_2
& = \sum_{n=0}^{2N+1}\frac{C_n}{(\log k)^n}
+ \frac{\varepsilon}{\sigma}\sum_{n=1}^{N+1}\frac{a_n+\sigma nC_{n-1}-C_n}{(\log k)^n}
\\
& \quad
+ \bigg(\frac{\varepsilon}{\sigma}-\frac{\varepsilon^2}{2\sigma^2}\bigg)(a_0-C_0) 
- \frac{\varepsilon^2}{2\sigma}C_0
+ o_{\sigma, N}(R_{2N+2}(k, y)).
\end{align*}
Choosing $a_0=C_0$ and $a_n=C_n-\sigma nC_{n-1}$ for $n\ge1$, we find that
\begin{align*}
\mathcal{S}_1 + \mathcal{S}_2
= \sum_{n=0}^{2N+1}\frac{C_n}{(\log k)^n} -\frac{\varepsilon^2}{2\sigma}C_0 + o_{\sigma, N}(R_{2N+2}(k, y)).
\end{align*}
Inserting this into (\ref{functionf}), and using (\ref{eq4.3}), then we have 
\begin{align*}
\int_{-\infty}^{\tau_2}{\rm e}^{kt} \Phi_T(t;y){\rm d}t
& \le \exp\bigg(\frac{k^{1/\sigma}}{\log k}\bigg\{-\frac{\varepsilon^2}{2\sigma}C_0
+ \sum_{n=0}^{2N+1}\frac{C_n}{(\log k)^n} + O_{\sigma, N}(R_{2N+2}(k, y))\bigg\}\bigg)
\\
& = \exp\bigg(\frac{k^{1/\sigma}}{\log k}\bigg\{-\frac{\varepsilon^2}{2\sigma}C_0 
+ O_{\sigma, N}(R_{2N+2}(k, y))\bigg\}\bigg)
\int_{-\infty}^{\infty}{\rm e}^{kt} \Phi_T(t;y){\rm d}t.
\end{align*}
By the choice of the value of $\varepsilon$ ($A=A(\sigma, N)$ is a suitably large constant), and $C_0>0$, 
we can obtain 
$$
\int_{-\infty}^{\tau_2}{\rm e}^{kt} \Phi_T(t;y){\rm d}t\le\frac{1}{4}\int_{-\infty}^{\infty}{\rm e}^{kt} \Phi_T(t;y){\rm d}t.
$$

Similarly, we have 
$$\int_{\tau_1}^{\infty}{\rm e}^{kt} \Phi_T(t;y){\rm d}t\le\frac{1}{4}\int_{-\infty}^{\infty}{\rm e}^{kt} \Phi_T(t;y){\rm d}t.$$
Thus combining the above two inequalities we have
$$
\frac{1}{2}\int_{-\infty}^{\infty}{\rm e}^{kt} \Phi_T(t;y){\rm d}t
\le \int_{\tau_2}^{\tau_1}{\rm e}^{kt} \Phi_T(t;y){\rm d}t
\le \int_{-\infty}^{\infty}{\rm e}^{kt} \Phi_T(t;y){\rm d}t.
$$
So thanks to \eqref{eq4.3},
we can get the asymptotic formula for the integral over $(\tau_2,\tau_1)$:
\begin{equation}\label{deffrakE}
\int_{\tau_2}^{\tau_1}{\rm e}^{kt} \Phi_T(t;y){\rm d}t
= \exp\bigg(\frac{k^{1/\sigma}}{\log k}\bigg\{\sum_{n=0}^{2N+1}\frac{C_n}{(\log k)^n}
+ O(R_{2N+2}(k, y))\bigg\}\bigg).
\end{equation}
On the other hand, since $\Phi_T(t; y)$ is decreasing in $t$, we have
$$
(\tau_1-\tau_2) {\rm e}^{k\tau_2} \Phi_T(\tau_1;y)
\le \int_{\tau_2}^{\tau_1}{\rm e}^{kt} \Phi_T(t;y){\rm d}t
\le (\tau_1-\tau_2) {\rm e}^{k\tau_1} \Phi_T(\tau_2;y).
$$
By the choice of the values of $\tau_1$ and $\tau_2$, the above inequality is
\begin{equation}\label{ineq4.10}
\frac{\varepsilon\tau}{\sigma}{\rm e}^{k\tau(1-\frac{\varepsilon}{2\sigma})}
\Phi_T((1+\tfrac{\varepsilon}{2\sigma})\tau; y)
\le \int_{\tau_2}^{\tau_1}{\rm e}^{kt} \Phi_T(t;y){\rm d}t
\le \frac{\varepsilon\tau}{\sigma}{\rm e}^{k\tau(1+\frac{\varepsilon}{2\sigma})}
\Phi_T((1-\tfrac{\varepsilon}{2\sigma})\tau; y).
\end{equation}
In view of \eqref{relation:tau=k}, it is easy to see that
\begin{align*}
\frac{\sigma}{\varepsilon\tau}{\rm e}^{-k\tau(1\pm\frac{\varepsilon}{2\sigma})}
& = \exp\bigg(\log\bigg(\frac{\sigma}{\varepsilon\tau}\bigg)-k\tau \{1+O(\varepsilon)\}\bigg)
\\
& = \exp\bigg(-\frac{k^{1/\sigma}}{\sigma\log k}
\bigg\{\sum_{n=0}^{N+1}\frac{a_n}{(\log k)^n}+O(\varepsilon)\bigg\}\bigg).
\end{align*}
Combining this with \eqref{deffrakE} and \eqref{ineq4.10}, it follows that
\begin{align*}
\frac{\sigma}{\varepsilon\tau}{\rm e}^{-k\tau(1\pm\frac{\varepsilon}{2\sigma})}
\int_{\tau_2}^{\tau_1}{\rm e}^{kt} \Phi_T(t;y){\rm d}t
= \exp\bigg(-\frac{k^{1/\sigma}}{\sigma\log k}\bigg\{\sum_{n=0}^{N}\frac{a_n-\sigma C_n}{(\log k)^n}
	+ O(\varepsilon)\bigg\}\bigg).
\end{align*}
Back to \eqref{ineq4.10}, we get
\begin{equation}\label{final:1}
\Phi_T((1+\tfrac{\varepsilon}{2\sigma})\tau; y)
\le \exp\bigg(-\frac{k^{1/\sigma}}{\sigma\log k}\bigg\{\sum_{n=0}^{N}\frac{a_n-\sigma C_n}{(\log k)^n}+ O(\varepsilon)\bigg\}\bigg)
\le \Phi_T((1-\tfrac{\varepsilon}{2\sigma})\tau; y).
\end{equation}
Recall that Lemma \ref{l822} and \eqref{k:tau} give
$$
\log k=\log\tau\bigg\{\sum_{n=0}^{N}\frac{b_n(\log_2\tau)}{(\log \tau)^n}+O\bigg(\bigg(\frac{\log_2\tau}{\log\tau}\bigg)^{N+1}\bigg)\bigg\}
$$
and
$$
\log k 
= \frac{\sigma}{1-\sigma}\log(\tau\log\tau)+O_{\sigma}(1).
$$
With the help of these formulas, after some computations of Taylor's expansions we easily see that
there are a sequence of polynomials $\{\mathfrak{a}_n(\cdot)\}_{n\ge 0}$
\footnote{The value of ${\mathfrak a}_0$ follows easily from $b_{0}=\frac{\sigma}{1-\sigma}$ in Lemma \ref{l822}.} with $\deg(\mathfrak{a}_n)\le n$
and a positive constant $c_0 = c_0(\sigma, N)$ depending on $(\sigma, N)$ such that
\begin{equation}\label{final:2}
\frac{k^{1/\sigma}}{\sigma\log k}\bigg\{\sum_{n=0}^{N}\frac{a_n-\sigma C_n}{(\log k)^n}+ O(\varepsilon)\bigg\}
= (\tau\log^{\sigma}\tau)^{\frac{1}{1-\sigma}}\bigg\{\sum_{n=0}^{N}\frac{\mathfrak{a}_n(\log_2\tau)}{(\log\tau)^n}+O(\varepsilon_0)\bigg\}
\end{equation}
and
\begin{equation}\label{final:3}
\varepsilon/(2\sigma)\le \varepsilon_0,
\end{equation}
where $\varepsilon_0$ is given as in \eqref{def:varepsilon0}.
Inserting \eqref{final:2} into \eqref{final:1} and 
using the fact that the function $t\mapsto \Phi_T(t; y)$ is decreasing with \eqref{final:3},
we obtain the required result.
This completes the proof.   
\hfill$\square$

\subsection{End of the proof of Theorem \ref{th22}}
Let 
$$
\eta
:= c_0\bigg(\frac{(\tau\log\tau)^{\frac{1}{1-\sigma}}}{y}\bigg)^{\sigma-\frac{1}{2}},
$$
where $c_0=c_0(\sigma, N)$ be a large positive constant given as in Proposition \ref{prop4.1}.
Applying Lemma \ref{l145} with $\lambda=\eta\tau$, we can obtain
\begin{equation}\label{EndProof:1}
\Phi_T(\tau) 
= \Phi_T(\tau(1\pm\eta);y)
+ O\bigg(\exp\bigg\{-(4{\rm e})^{-1}(\sigma-\tfrac{1}{2})c_0^2 \frac{\log y}{\log\tau}(\tau\log^{\sigma}\tau)^{\frac{1}{1-\sigma}}\bigg\}\bigg).
\end{equation}
On the other hand, 
noticing that $\eta\le \varepsilon_0$ and that $\Phi_T(t; y)$ is decreasing in $t$,
\eqref{EndProof:1} and Proposition \ref{prop4.1} imply that
\begin{equation}\label{EndProof:2}
\begin{aligned}
\Phi_T(\tau) 
& = \exp\bigg(-(\tau \log^{\sigma}\tau)^{\frac{1}{1-\sigma}}
\bigg\{\sum_{n=0}^{N}\frac{\mathfrak{a}_n(\log_2\tau)}{(\log\tau)^n} + O(\varepsilon_0)\bigg\}
\bigg)
\\
& \quad
+ O\bigg(\exp\bigg\{-(4{\rm e})^{-1}(\sigma-\tfrac{1}{2})c_0^2 \frac{\log y}{\log\tau}(\tau\log^{\sigma}\tau)^{\frac{1}{1-\sigma}}\bigg\}\bigg)
\\
& = \exp\bigg(-(\tau \log^{\sigma}\tau)^{\frac{1}{1-\sigma}}
\bigg\{\sum_{n=0}^{N}\frac{\mathfrak{a}_n(\log_2\tau)}{(\log\tau)^n} + O(\varepsilon_0)\bigg\}
\bigg) \Delta(\tau, y),
\end{aligned}
\end{equation}
uniformly for 
$$
T\to\infty,
\quad
\log T\le y\le(\log T)^2,
\quad 
1\ll\tau\le  { \frac{c(\sigma)}{\log_2T}\bigg(\frac{\log T}{y^{1-\sigma}}\bigg)^{\frac{1-\sigma}{\sigma}}},
$$
where
$$
 \Delta(\tau, y)
 := 1 
 + O\bigg(\exp\bigg\{-(4{\rm e})^{-1}(\sigma-\tfrac{1}{2})c_0^2 \frac{\log y}{\log\tau}(\tau\log^{\sigma}\tau)^{\frac{1}{1-\sigma}} 
 + O_{\sigma, N}\big((\tau\log^{\sigma}\tau)^{\frac{1}{1-\sigma}}\big)
 \bigg\}\bigg).
$$
Since $c_0$ is suitably large, we have, with choice of $y=\log T$,
\begin{align*}
\Delta(\tau, \log T)
& = 1 
+ O\big(\exp\big\{-(8{\rm e})^{-1}(\sigma-\tfrac{1}{2})c_0^2(\tau\log^{\sigma}\tau)^{\frac{1}{1-\sigma}}\big\}\big)
\\
& = \exp\big(\exp\big\{-(8{\rm e})^{-1}(\sigma-\tfrac{1}{2})c_0^2(\tau\log^{\sigma}\tau)^{\frac{1}{1-\sigma}}\big\}\big)
\\
& = \exp\big\{O\big((\tau\log^{\sigma}\tau)^{\frac{1}{1-\sigma}}\varepsilon_0(\tau, \log T)\big)\big\}
\end{align*}
uniformly for 
$T\to\infty$ and $1\ll\tau\le c(\sigma)(\log T)^{1-\sigma}/\log_2T$.
Inserting this into \eqref{EndProof:2}, we obtain the result of Theorem \ref{th22}.
\hfill
$\square$

\vskip 5mm

\textbf{Acknowledgement}.
The author would like to thank professor Jie Wu for his suggestion on exploring this subject, and Masahiro Mine for his valuable remarks on Lemma \ref{l822} and pointing to his work \cite{MM20}. The author is also grateful to Jinjiang Li for helping correct some typos. The author is supported by the China Scholarship Council (CSC) for his study in France.

\end{document}